\documentclass[10pt,a4paper]{article}
\usepackage{tikz}
\usepackage[width=14.00cm, height=25.00cm]{geometry}
\usepackage[T1]{fontenc}
\usepackage[utf8]{inputenc}
\usepackage{amsmath,amssymb,amsfonts,amsthm,amscd, bm, mathtools}
\usepackage{enumerate}
\usepackage{todonotes}
\usepackage{tikz}
\usepackage{hyperref}
\usetikzlibrary{positioning}
\usetikzlibrary{decorations.pathmorphing}
\usetikzlibrary{decorations.text}
\newtheorem{thm}{Theorem}[section]
\newtheorem{prop}[thm]{Proposition}
\newtheorem{lemma}[thm]{Lemma}

\newtheorem{defn}[thm]{Definition}

\newtheorem{ex}[thm]{Example}
\newtheorem{claim}[thm]{Claim}

\title{Edge-ends versus topological ends of graphs}
\author{Leandro Aurichi, Paulo Magalhães Júnior and Guilherme Eduardo Pinto}

\newcommand{\Addresses}{{
  \bigskip
  \footnotesize
  L. ~Aurichi, \textsc{Instituto de Ci\^encias Matem\'aticas e de Computa\c c\~ao, Universidade de S\~ao Paulo\\
	Avenida Trabalhador s\~ao-carlense, 400,  S\~ao Carlos, SP, 13566-590, Brazil}\par\nopagebreak
  \textit{E-mail address}, L.~Aurichi: \texttt{aurichi@icmc.usp.br}

  \medskip
  P.~Magalhães Jr, \textsc{Departamento de Matemática - Universidade Federal do Piauí\\
	Avenida Nossa Senhora de Fátima, s/n, bairro Ininga, Teresina, PI, 64.049-550, Brazil}\par\nopagebreak
  \textit{E-mail address}, P.~Magalhães Jr: \texttt{pjr.mat@gmail.com}
  
  \medskip
  G.~Pinto, \textsc{Instituto de Ci\^encias Matem\'aticas e de Computa\c c\~ao, Universidade de S\~ao Paulo\\
	Avenida Trabalhador s\~ao-carlense, 400,  S\~ao Carlos, SP, 13566-590, Brazil}\par\nopagebreak
  \textit{E-mail address}, G.~Pinto: \texttt{guipullus.gp@usp.br}
 
}}
\date{}

\begin{document}

\maketitle
\begin{abstract}
Diestel and Kühn proved that the topological ends of an infinite graph are precisely its undominated graph ends, yielding a canonical embedding of the space of topological ends into the space of graph ends. For edge-ends, introduced by Hahn, Laviolette and Širáň, such an embedding does not exist in general.

In this note, we characterize the class of infinite graphs for which the topological ends admit a natural injective map into the space of edge-ends that is compatible with the canonical maps between end spaces. Our characterization is purely combinatorial and is expressed in terms of edge-equivalence classes of vertices.

Moreover, when such an embedding exists, we identify precisely which edge-ends arise from topological ends, showing that they are exactly the edge-ends containing a non-dominated ray. This establishes a parallel result to the theorem of Diestel and Kühn for edge-end spaces. 
\end{abstract}

\section{Introduction}

\paragraph{}
A ray $r$ in an infinite graph $G$ is a one-way infinite path. A tail of a ray $r$ is an infinite, connected subgraph of $r$. We say that two rays $r$ and $s$ of $G$ are equivalent, and we denote it by $r\sim s$, if for every finite set of vertices $X\subset V(G)$, $r$ and $s$ have tails in the same connected component of $G\setminus X$. This relation is an equivalence relation on the set of rays of $G$, denoted by $\mathcal{R}(G)$. In \cite{halin}, Halin defined the ends of a graph as the equivalence classes with respect to $\sim$. The set of equivalence classes, $\Omega(G)=\frac{\mathcal{R}(G)}{\sim}=\lbrace [r]: r\in \mathcal{R}(G)\rbrace$, is called the end space. There is a natural topology with which we can equip the space $\Omega(G)$; for more on this topology, see \cite{halin}.

In \cite{Freudenthal}, Freudenthal introduced the concept of ends for certain topological spaces as points at infinity that aim to compactify these spaces. More formally, we define topological ends as follows: let $X$ be a Hausdorff topological space and $A_0\supset A_1 \supset A_2 \supset \cdots \supset A_n \supset \cdots$, a decreasing chain of open sets of $X$, with the boundary of each $A_n$, $\partial A_n$, is a compact set and $\bigcap_{n\in\mathbb{N}}\overline{A_n}=\emptyset$. We say that two sequences with these properties, $\langle A_n: n\in\mathbb{N}\rangle$ and $\langle B_n: n\in\mathbb{N}\rangle$, are equivalent if for every $n\in\mathbb{N}$ there exist $m,k\in\mathbb{N}$ such that $A_n\supset B_m$ and $B_n\supset A_k$. The equivalence classes of these sequences are the topological ends of $X$. The set of all topological ends of $X$ is denoted by $\mathcal{C}(X)$. If $\langle A_n: n\in\mathbb{N}\rangle$ is a sequence contained in an end $\epsilon$, we say that $\langle A_n:n\in\mathbb{N}\rangle$ represents $\epsilon$.

If we consider an infinite graph $G$, equipped with the topology of the 1-simplex, then for locally finite graphs these two definitions coincide. For more general graphs, this is not the case. We say that a vertex $v\in V(G)$ dominates a ray $r$ of $G$ if there are infinitely many paths from $v$ to the ray $r$ that are pairwise disjoint except at $v$. In \cite{topologicalendsversuscombinatorialends}, Diestel and Kühn showed that the topological ends of a graph ($\Omega'(G)$) are in bijection with the non-dominated graph ends. That is, they showed that there exists a natural injective function that maps a topological end $\epsilon$ of $G$ to a non-dominated end of $G$, $\varepsilon$.

$$\begin{array}{ccccc}
   \phi:& \Omega'(G) & \longrightarrow & \Omega(G)  \\
    & \epsilon  & \mapsto & \varepsilon
\end{array}$$

Thus there always exists a copy of $\Omega'(G)$ in $\Omega(G)$. Results that establish a relationship between combinatorial structures and the topological ends of a graph are of great relevance. Other works along this line can be found, for example, in \cite{carmesintreedecompositiontopologicalends, pitztreedecompositiontopologicalends}.

We say that two rays $r$ and $s$ of $G$ are edge-equivalent, and we denote it by $r\sim_E s$, if for every finite set of edges $X\subset E(G)$, $r$ and $s$ have tails in the same connected component of $G\setminus X$. This relation is an equivalence relation on the set of rays of $G$. In \cite{edge-ends}, Hahn, Laviolette, and Širáň defined the edge-ends of a graph as the equivalence classes with respect to $\sim_E$. The set of equivalence classes, $\Omega_E(G)=\frac{\mathcal{R}(G)}{\sim_E}=\lbrace [r]_E: r\in\mathcal{R}(G)\rbrace$, is called the edge-end space. There is a natural topology that we can equip the space $\Omega_E(G)$ with, similar to the topology of $\Omega(G)$; for more on this topology, see \cite{aurichi2024topologicalremarksendedgeend, edge-ends}.
Although the definition of an edge-end was introduced in 1997, only recently has there been the beginning of a great interest in studies on these spaces, such as, for example, \cite{aurichi2024topologicalremarksendedgeend, aurichilucasedgeconectividade, boska2025edgedirectioncompactedgeendspaces, pitz2025metrizationtheoremedgeendspaces, real2025subbasepropertydescribingedgeend}.

Note that, in any graph $G$, we always have a natural surjective function from $\Omega(G)$ to $\Omega_E(G)$, which maps an end $\varepsilon=[r]$ to the edge-end $\varepsilon'=[r]_E$.

$$\begin{array}{ccccc}
   f:& \Omega(G) & \longrightarrow & \Omega_E(G)  \\
    & [r]  & \mapsto & [r]_E
\end{array}$$

Just as a natural function is established between
$\Omega'(G)$ and $\Omega(G)$ in \cite{topologicalendsversuscombinatorialends}, we find that it is possible to establish a natural function between $\Omega'(G)$ and $\Omega_E(G)$. A simple way to do this is to consider the function $f_E=f\circ \phi$. Unlike end spaces, it is not always possible to "see" the topological ends of $G$ inside the edge-ends of $G$, that is, it is not always possible to obtain a natural injective function from the topological ends of a graph $G$ to the edge-ends of $G$. In other words, given a graph $G$, it is not always possible for $f_E=f\circ \phi$ to be injective and to make Diagram \ref{diagrama 1} commute. See the example below.

\begin{figure}[ht]
            \centering

\tikzset{every picture/.style={line width=0.75pt}} %set default line width to 0.75pt        

\begin{tikzpicture}[x=0.6pt,y=0.6pt,yscale=-1,xscale=1]

\draw    (254.5,51.6) -- (300.9,86.4) ;
\draw [shift={(302.5,87.6)}, rotate = 216.87] [color={rgb, 255:red, 0; green, 0; blue, 0 }  ][line width=0.75]    (10.93,-3.29) .. controls (6.95,-1.4) and (3.31,-0.3) .. (0,0) .. controls (3.31,0.3) and (6.95,1.4) .. (10.93,3.29)   ;
%Straight Lines [id:da1253673324609177] 
\draw    (274.5,33.6) -- (376.5,32.62) ;
\draw [shift={(378.5,32.6)}, rotate = 179.45] [color={rgb, 255:red, 0; green, 0; blue, 0 }  ][line width=0.75]    (10.93,-3.29) .. controls (6.95,-1.4) and (3.31,-0.3) .. (0,0) .. controls (3.31,0.3) and (6.95,1.4) .. (10.93,3.29)   ;

\draw    (356.5,89.6) -- (403.95,50.86) ;
\draw [shift={(405.5,49.6)}, rotate = 140.77] [color={rgb, 255:red, 0; green, 0; blue, 0 }  ][line width=0.75]    (10.93,-3.29) .. controls (6.95,-1.4) and (3.31,-0.3) .. (0,0) .. controls (3.31,0.3) and (6.95,1.4) .. (10.93,3.29)   ;

\draw (311,81.4) node [anchor=north west][inner sep=0.75pt]    {$\Omega ( G)$};

\draw (225,25.4) node [anchor=north west][inner sep=0.75pt]    {$\Omega '( G)$};

\draw (388,26.4) node [anchor=north west][inner sep=0.75pt]    {$\Omega _{E}( G)$};

\draw (311,8.4) node [anchor=north west][inner sep=0.75pt]    {$f_{E}$};

\draw (386,67.4) node [anchor=north west][inner sep=0.75pt]    {$f$};

\draw (252,65.4) node [anchor=north west][inner sep=0.75pt]    {$\phi $};

\end{tikzpicture}

     \vspace{0.5cm} \text{Diagram 1} 
    \label{diagrama 1}
\end{figure}

\begin{ex}
Consider the graph $G$ from Figure \ref{não injetivo}. Note that $\vert \Omega (G)\vert =2$, since it is enough to remove any pair of vertices that are connected by a vertical edge to separate the graph. Furthermore, since there are no dominated rays in the graph $G$, we have $\Omega'(G)=\Omega(G)$ and, therefore, $\vert \Omega'(G)\vert =2$. On the other hand, note that $\vert \Omega_E(G)\vert = 1$, since all rays are edge-equivalent, as the vertices on the top part are all infinitely edge-connected. Thus, it is not possible to obtain an injective function from $\Omega'(G)$ into $\Omega_E(G)$. In particular, it is not possible to obtain an injective function $f_E:\Omega'(G)\longrightarrow \Omega_E(G)$ such that Diagram \ref{diagrama 1} commutes.
 
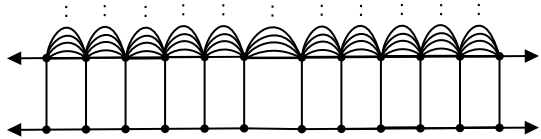
\begin{figure}[ht]
    \centering

\tikzset{every picture/.style={line width=0.75pt}}  

\begin{tikzpicture}[x=0.6pt,y=0.6pt,yscale=-1,xscale=1]

\draw    (168,41.98) -- (189.67,41.83) ;
\draw [shift={(165,42)}, rotate = 359.61] [fill={rgb, 255:red, 0; green, 0; blue, 0 }  ][line width=0.08]  [draw opacity=0] (8.93,-4.29) -- (0,0) -- (8.93,4.29) -- cycle    ;

\draw    (189.67,41.83) -- (214.33,41.67) ;
\draw [shift={(214.33,41.67)}, rotate = 359.61] [color={rgb, 255:red, 0; green, 0; blue, 0 }  ][fill={rgb, 255:red, 0; green, 0; blue, 0 }  ][line width=0.75]      (0, 0) circle [x radius= 2.01, y radius= 2.01]   ;
\draw [shift={(189.67,41.83)}, rotate = 359.61] [color={rgb, 255:red, 0; green, 0; blue, 0 }  ][fill={rgb, 255:red, 0; green, 0; blue, 0 }  ][line width=0.75]      (0, 0) circle [x radius= 2.01, y radius= 2.01]   ;

\draw    (214.33,41.67) -- (239,41.5) ;

\draw    (239,41.5) -- (263.67,41.33) ;
\draw [shift={(239,41.5)}, rotate = 359.61] [color={rgb, 255:red, 0; green, 0; blue, 0 }  ][fill={rgb, 255:red, 0; green, 0; blue, 0 }  ][line width=0.75]      (0, 0) circle [x radius= 2.01, y radius= 2.01]   ;

\draw    (263.67,41.33) -- (288.33,41.17) ;
\draw [shift={(288.33,41.17)}, rotate = 359.61] [color={rgb, 255:red, 0; green, 0; blue, 0 }  ][fill={rgb, 255:red, 0; green, 0; blue, 0 }  ][line width=0.75]      (0, 0) circle [x radius= 2.01, y radius= 2.01]   ;
\draw [shift={(263.67,41.33)}, rotate = 359.61] [color={rgb, 255:red, 0; green, 0; blue, 0 }  ][fill={rgb, 255:red, 0; green, 0; blue, 0 }  ][line width=0.75]      (0, 0) circle [x radius= 2.01, y radius= 2.01]   ;
 
\draw    (288.33,41.17) -- (313,41) ;
\draw [shift={(313,41)}, rotate = 359.61] [color={rgb, 255:red, 0; green, 0; blue, 0 }  ][fill={rgb, 255:red, 0; green, 0; blue, 0 }  ][line width=0.75]      (0, 0) circle [x radius= 2.01, y radius= 2.01]   ;
 
\draw    (168,86.98) -- (189.67,86.83) ;
\draw [shift={(165,87)}, rotate = 359.61] [fill={rgb, 255:red, 0; green, 0; blue, 0 }  ][line width=0.08]  [draw opacity=0] (8.93,-4.29) -- (0,0) -- (8.93,4.29) -- cycle    ;

\draw    (189.67,86.83) -- (214.33,86.67) ;
\draw [shift={(214.33,86.67)}, rotate = 359.61] [color={rgb, 255:red, 0; green, 0; blue, 0 }  ][fill={rgb, 255:red, 0; green, 0; blue, 0 }  ][line width=0.75]      (0, 0) circle [x radius= 2.01, y radius= 2.01]   ;
\draw [shift={(189.67,86.83)}, rotate = 359.61] [color={rgb, 255:red, 0; green, 0; blue, 0 }  ][fill={rgb, 255:red, 0; green, 0; blue, 0 }  ][line width=0.75]      (0, 0) circle [x radius= 2.01, y radius= 2.01]   ;

\draw    (214.33,86.67) -- (239,86.5) ;

\draw    (239,86.5) -- (263.67,86.33) ;
\draw [shift={(239,86.5)}, rotate = 359.61] [color={rgb, 255:red, 0; green, 0; blue, 0 }  ][fill={rgb, 255:red, 0; green, 0; blue, 0 }  ][line width=0.75]      (0, 0) circle [x radius= 2.01, y radius= 2.01]   ;
 
\draw    (263.67,86.33) -- (288.33,86.17) ;
\draw [shift={(288.33,86.17)}, rotate = 359.61] [color={rgb, 255:red, 0; green, 0; blue, 0 }  ][fill={rgb, 255:red, 0; green, 0; blue, 0 }  ][line width=0.75]      (0, 0) circle [x radius= 2.01, y radius= 2.01]   ;
\draw [shift={(263.67,86.33)}, rotate = 359.61] [color={rgb, 255:red, 0; green, 0; blue, 0 }  ][fill={rgb, 255:red, 0; green, 0; blue, 0 }  ][line width=0.75]      (0, 0) circle [x radius= 2.01, y radius= 2.01]   ;

\draw    (288.33,86.17) -- (313,86) ;
\draw [shift={(313,86)}, rotate = 359.61] [color={rgb, 255:red, 0; green, 0; blue, 0 }  ][fill={rgb, 255:red, 0; green, 0; blue, 0 }  ][line width=0.75]      (0, 0) circle [x radius= 2.01, y radius= 2.01]   ;

\draw    (313,41) -- (313,86) ;
 
\draw    (288.33,41.17) -- (288.33,86.17) ;
 
\draw    (263.67,41.33) -- (263.67,86.33) ;

\draw    (239,41.5) -- (239,86.5) ;
 
\draw    (214.33,41.67) -- (214.33,86.67) ;
 
\draw    (189.67,41.83) -- (189.67,86.83) ;
 
\draw    (288.33,41.17) .. controls (289,40.5) and (300,31) .. (313,41) ;
 
\draw    (288.33,41.17) .. controls (287.5,37) and (303,21) .. (313,41) ;
 
\draw    (288.33,41.17) .. controls (286,34.5) and (304.5,12.5) .. (313,41) ;
 
\draw    (288.33,41.17) .. controls (286,34.5) and (302,1) .. (313,41) ;
 
\draw  [dash pattern={on 0.84pt off 2.51pt}]  (300,8) -- (300,18) ;
 
\draw    (263.33,41.17) -- (288,41) ;
 
\draw    (263.33,41.17) .. controls (264,40.5) and (275,31) .. (288,41) ;
 
\draw    (263.33,41.17) .. controls (262.5,37) and (278,21) .. (288,41) ;

\draw    (263.33,41.17) .. controls (261,34.5) and (279.5,12.5) .. (288,41) ;
 
\draw    (263.33,41.17) .. controls (261,34.5) and (277,1) .. (288,41) ;
 
\draw  [dash pattern={on 0.84pt off 2.51pt}]  (275,8) -- (275,18) ;

\draw    (239.83,42.17) -- (264.5,42) ;
 
\draw    (239.83,42.17) .. controls (240.5,41.5) and (251.5,32) .. (264.5,42) ;
 
\draw    (239.83,42.17) .. controls (239,38) and (254.5,22) .. (264.5,42) ;
 
\draw    (239.83,42.17) .. controls (237.5,35.5) and (256,13.5) .. (264.5,42) ;

\draw    (239.83,42.17) .. controls (237.5,35.5) and (253.5,2) .. (264.5,42) ;
 
\draw  [dash pattern={on 0.84pt off 2.51pt}]  (251.5,9) -- (251.5,19) ;

\draw    (214.33,41.67) -- (239,41.5) ;

\draw    (214.33,41.67) .. controls (215,41) and (226,31.5) .. (239,41.5) ;
 
\draw    (214.33,41.67) .. controls (213.5,37.5) and (229,21.5) .. (239,41.5) ;

\draw    (214.33,41.67) .. controls (212,35) and (230.5,13) .. (239,41.5) ;
 
\draw    (214.33,41.67) .. controls (212,35) and (228,1.5) .. (239,41.5) ;
 
\draw  [dash pattern={on 0.84pt off 2.51pt}]  (226,8.5) -- (226,18.5) ;
 
\draw    (190.33,42.17) -- (215,42) ;

\draw    (190.33,42.17) .. controls (191,41.5) and (202,32) .. (215,42) ;

\draw    (190.33,42.17) .. controls (189.5,38) and (205,22) .. (215,42) ;
 
\draw    (190.33,42.17) .. controls (188,35.5) and (206.5,13.5) .. (215,42) ;

\draw    (190.33,42.17) .. controls (188,35.5) and (204,2) .. (215,42) ;
 
\draw  [dash pattern={on 0.84pt off 2.51pt}]  (202,9) -- (202,19) ;

\draw    (494,85.52) -- (472.33,85.67) ;
\draw [shift={(497,85.5)}, rotate = 179.61] [fill={rgb, 255:red, 0; green, 0; blue, 0 }  ][line width=0.08]  [draw opacity=0] (8.93,-4.29) -- (0,0) -- (8.93,4.29) -- cycle    ;
 
\draw    (472.33,85.67) -- (447.67,85.83) ;
\draw [shift={(447.67,85.83)}, rotate = 179.61] [color={rgb, 255:red, 0; green, 0; blue, 0 }  ][fill={rgb, 255:red, 0; green, 0; blue, 0 }  ][line width=0.75]      (0, 0) circle [x radius= 2.01, y radius= 2.01]   ;
\draw [shift={(472.33,85.67)}, rotate = 179.61] [color={rgb, 255:red, 0; green, 0; blue, 0 }  ][fill={rgb, 255:red, 0; green, 0; blue, 0 }  ][line width=0.75]      (0, 0) circle [x radius= 2.01, y radius= 2.01]   ;
 
\draw    (447.67,85.83) -- (423,86) ;

\draw    (423,86) -- (398.33,86.17) ;
\draw [shift={(423,86)}, rotate = 179.61] [color={rgb, 255:red, 0; green, 0; blue, 0 }  ][fill={rgb, 255:red, 0; green, 0; blue, 0 }  ][line width=0.75]      (0, 0) circle [x radius= 2.01, y radius= 2.01]   ;
 
\draw    (398.33,86.17) -- (373.67,86.33) ;
\draw [shift={(373.67,86.33)}, rotate = 179.61] [color={rgb, 255:red, 0; green, 0; blue, 0 }  ][fill={rgb, 255:red, 0; green, 0; blue, 0 }  ][line width=0.75]      (0, 0) circle [x radius= 2.01, y radius= 2.01]   ;
\draw [shift={(398.33,86.17)}, rotate = 179.61] [color={rgb, 255:red, 0; green, 0; blue, 0 }  ][fill={rgb, 255:red, 0; green, 0; blue, 0 }  ][line width=0.75]      (0, 0) circle [x radius= 2.01, y radius= 2.01]   ;
 
\draw    (373.67,86.33) -- (349,86.5) ;
\draw [shift={(349,86.5)}, rotate = 179.61] [color={rgb, 255:red, 0; green, 0; blue, 0 }  ][fill={rgb, 255:red, 0; green, 0; blue, 0 }  ][line width=0.75]      (0, 0) circle [x radius= 2.01, y radius= 2.01]   ;
 
\draw    (494,40.52) -- (472.33,40.67) ;
\draw [shift={(497,40.5)}, rotate = 179.61] [fill={rgb, 255:red, 0; green, 0; blue, 0 }  ][line width=0.08]  [draw opacity=0] (8.93,-4.29) -- (0,0) -- (8.93,4.29) -- cycle    ;
 
\draw    (472.33,40.67) -- (447.67,40.83) ;
\draw [shift={(447.67,40.83)}, rotate = 179.61] [color={rgb, 255:red, 0; green, 0; blue, 0 }  ][fill={rgb, 255:red, 0; green, 0; blue, 0 }  ][line width=0.75]      (0, 0) circle [x radius= 2.01, y radius= 2.01]   ;
\draw [shift={(472.33,40.67)}, rotate = 179.61] [color={rgb, 255:red, 0; green, 0; blue, 0 }  ][fill={rgb, 255:red, 0; green, 0; blue, 0 }  ][line width=0.75]      (0, 0) circle [x radius= 2.01, y radius= 2.01]   ;

\draw    (447.67,40.83) -- (423,41) ;
 
\draw    (423,41) -- (398.33,41.17) ;
\draw [shift={(423,41)}, rotate = 179.61] [color={rgb, 255:red, 0; green, 0; blue, 0 }  ][fill={rgb, 255:red, 0; green, 0; blue, 0 }  ][line width=0.75]      (0, 0) circle [x radius= 2.01, y radius= 2.01]   ;
 
\draw    (398.33,41.17) -- (373.67,41.33) ;
\draw [shift={(373.67,41.33)}, rotate = 179.61] [color={rgb, 255:red, 0; green, 0; blue, 0 }  ][fill={rgb, 255:red, 0; green, 0; blue, 0 }  ][line width=0.75]      (0, 0) circle [x radius= 2.01, y radius= 2.01]   ;
\draw [shift={(398.33,41.17)}, rotate = 179.61] [color={rgb, 255:red, 0; green, 0; blue, 0 }  ][fill={rgb, 255:red, 0; green, 0; blue, 0 }  ][line width=0.75]      (0, 0) circle [x radius= 2.01, y radius= 2.01]   ;
 
\draw    (373.67,41.33) -- (349,41.5) ;
\draw [shift={(349,41.5)}, rotate = 179.61] [color={rgb, 255:red, 0; green, 0; blue, 0 }  ][fill={rgb, 255:red, 0; green, 0; blue, 0 }  ][line width=0.75]      (0, 0) circle [x radius= 2.01, y radius= 2.01]   ;
 
\draw    (349,86.5) -- (349,41.5) ;
 
\draw    (373.67,86.33) -- (373.67,41.33) ;
 
\draw    (398.33,86.17) -- (398.33,41.17) ;
 
\draw    (423,86) -- (423,41) ;
%Straight Lines [id:da7783286895302717] 
\draw    (447.67,85.83) -- (447.67,40.83) ;
%Straight Lines [id:da5999287868730145] 
\draw    (472.33,85.67) -- (472.33,40.67) ;
%Straight Lines [id:da546350566024382] 
\draw    (398.67,86.33) -- (374,86.5) ;
%Straight Lines [id:da8104247424241364] 
\draw    (422.17,85.33) -- (397.5,85.5) ;
%Straight Lines [id:da9442097994678573] 
\draw    (447.67,85.83) -- (423,86) ;
%Straight Lines [id:da4250012304998426] 
\draw    (471.67,85.33) -- (447,85.5) ;
%Straight Lines [id:da18112895539287022] 
\draw    (348.93,41.05) -- (373.6,40.99) ;
%Straight Lines [id:da8342013664350961] 
\draw    (373.6,40.99) -- (398.27,40.93) ;
%Straight Lines [id:da8479569953232728] 
\draw    (373.6,40.99) -- (398.27,40.93) ;
%Curve Lines [id:da17699923933300632] 
\draw    (373.6,40.99) .. controls (374.27,40.33) and (385.31,30.88) .. (398.27,40.93) ;
%Curve Lines [id:da5832777475255654] 
\draw    (373.6,40.99) .. controls (372.78,36.82) and (388.35,20.89) .. (398.27,40.93) ;
%Curve Lines [id:da07990761679489666] 
\draw    (373.6,40.99) .. controls (371.3,34.32) and (389.89,12.4) .. (398.27,40.93) ;
%Curve Lines [id:da8907353156866532] 
\draw    (373.6,40.99) .. controls (371.3,34.32) and (387.44,0.89) .. (398.27,40.93) ;
%Straight Lines [id:da8478187515742643] 
\draw    (349.6,41.39) -- (374.27,41.33) ;
%Curve Lines [id:da286410505126133] 
\draw    (349.6,41.39) .. controls (350.27,40.72) and (361.31,31.27) .. (374.27,41.33) ;
%Curve Lines [id:da9453275661655898] 
\draw    (349.6,41.39) .. controls (348.78,37.22) and (364.35,21.29) .. (374.27,41.33) ;
%Curve Lines [id:da8242828139748426] 
\draw    (349.6,41.39) .. controls (347.29,34.71) and (365.89,12.79) .. (374.27,41.33) ;
%Curve Lines [id:da6246710712706535] 
\draw    (349.6,41.39) .. controls (347.29,34.71) and (363.44,1.28) .. (374.27,41.33) ;
%Straight Lines [id:da2886139177953734] 
\draw  [dash pattern={on 0.84pt off 2.51pt}]  (361.41,8.27) -- (361.37,18.27) ;
%Straight Lines [id:da17800927419268542] 
\draw  [dash pattern={on 0.84pt off 2.51pt}]  (385.91,8.38) -- (385.87,18.38) ;
%Straight Lines [id:da06666678091412981] 
\draw    (398.93,40.55) -- (423.6,40.49) ;
%Straight Lines [id:da45816623679329616] 
\draw    (423.6,40.49) -- (448.27,40.43) ;
%Straight Lines [id:da3225319640663805] 
\draw    (423.6,40.49) -- (448.27,40.43) ;
%Curve Lines [id:da4652680534650874] 
\draw    (423.6,40.49) .. controls (424.27,39.83) and (435.31,30.38) .. (448.27,40.43) ;
%Curve Lines [id:da6132150559879036] 
\draw    (423.6,40.49) .. controls (422.78,36.32) and (438.35,20.39) .. (448.27,40.43) ;
%Curve Lines [id:da1575460153198769] 
\draw    (423.6,40.49) .. controls (421.3,33.82) and (439.89,11.9) .. (448.27,40.43) ;
%Curve Lines [id:da4079489543806507] 
\draw    (423.6,40.49) .. controls (421.3,33.82) and (437.44,0.39) .. (448.27,40.43) ;
%Straight Lines [id:da2880636497958602] 
\draw    (399.6,40.89) -- (424.27,40.83) ;
%Curve Lines [id:da708096573398709] 
\draw    (399.6,40.89) .. controls (400.27,40.22) and (411.31,30.77) .. (424.27,40.83) ;
%Curve Lines [id:da18233177919761656] 
\draw    (399.6,40.89) .. controls (398.78,36.72) and (414.35,20.79) .. (424.27,40.83) ;
%Curve Lines [id:da545693405136769] 
\draw    (399.6,40.89) .. controls (397.29,34.21) and (415.89,12.29) .. (424.27,40.83) ;
%Curve Lines [id:da8588634127553535] 
\draw    (399.6,40.89) .. controls (397.29,34.21) and (413.44,0.78) .. (424.27,40.83) ;
%Straight Lines [id:da2747582300695437] 
\draw  [dash pattern={on 0.84pt off 2.51pt}]  (411.41,7.77) -- (411.37,17.77) ;
%Straight Lines [id:da40711354729171434] 
\draw  [dash pattern={on 0.84pt off 2.51pt}]  (435.91,7.88) -- (435.87,17.88) ;
%Straight Lines [id:da6970796045744452] 
\draw    (446.93,40.55) -- (471.6,40.49) ;
%Straight Lines [id:da7390130061048283] 
\draw    (447.6,40.89) -- (472.27,40.83) ;
%Curve Lines [id:da7060344771689393] 
\draw    (447.6,40.89) .. controls (448.27,40.22) and (459.31,30.77) .. (472.27,40.83) ;
%Curve Lines [id:da8923146640880121] 
\draw    (447.6,40.89) .. controls (446.78,36.72) and (462.35,20.79) .. (472.27,40.83) ;
%Curve Lines [id:da9453836245603933] 
\draw    (447.6,40.89) .. controls (445.29,34.21) and (463.89,12.29) .. (472.27,40.83) ;
%Curve Lines [id:da12670662534539168] 
\draw    (447.6,40.89) .. controls (445.29,34.21) and (461.44,0.78) .. (472.27,40.83) ;
 
\draw  [dash pattern={on 0.84pt off 2.51pt}]  (459.41,7.77) -- (459.37,17.77) ;
 
\draw    (313.43,41.55) -- (347.58,41.49) ;

\draw    (314.35,41.89) -- (348.5,41.83) ;
 
\draw    (314.35,41.89) .. controls (315.28,41.22) and (330.56,31.77) .. (348.5,41.83) ;
 
\draw    (314.35,41.89) .. controls (313.23,37.72) and (334.78,21.79) .. (348.5,41.83) ;

\draw    (314.35,41.89) .. controls (311.16,35.21) and (336.91,13.29) .. (348.5,41.83) ;
 
\draw    (314.35,41.89) .. controls (311.16,35.21) and (333.51,1.78) .. (348.5,41.83) ;

\draw  [dash pattern={on 0.84pt off 2.51pt}]  (330.7,8.77) -- (330.64,18.77) ;
 
\draw    (313,86) -- (349,86.5) ;
\end{tikzpicture}

    \caption{A graph where it is not possible to obtain an injection from $\Omega'(G)$ into $\Omega_E(G)$.}
   
    \label{não injetivo}
\end{figure}

\end{ex}

In this note, we are interested in combinatorially characterizing the graphs for which it is possible to obtain an injective function $f_E$ that makes Diagram \ref{diagrama 1} commute, that is, the graphs $G$ in which it is possible to obtain a copy of the topological ends of $G$ inside the edge-end space of $G$. The main result of this paper is Theorem \ref{theo1}, in which we provide a combinatorial characterization for the presented problem. Furthermore, we proved which edge-ends of a graph $G$ are topological ends of this graph, establishing a result parallel to the main result of \cite{topologicalendsversuscombinatorialends}.

\section{The characterization}
\paragraph{}
In this section, we will present and prove a characterization theorem for infinite graphs $G$ in which it is possible to obtain an injective function $f_E:\Omega'(G)\longrightarrow \Omega_E(G)$ that makes Diagram \ref{diagrama 1} commute. To do this, we will need the concept of the edge-equivalence class (or edge class) of a vertex $v$ in a graph $G$.

\begin{defn}
Let $G$ be a graph and $v,u\in V(G)$. We say that $v$ and $u$ are edge-equivalent, and we denote it by $u\sim_E v$, if for every finite set of edges $F\subset E(G)$, $u$ and $v$ are in the same connected component of $G\setminus F$.

\end{defn}

Note that $u\sim_E v$ is equivalent to $u$ and $v$ being infinitely edge-connected. The relation $\sim_E$ is an equivalence relation on the vertices of a graph $G$. This concept of edge-equivalent vertices appears in some works in the literature; see, for example, \cite{boska2025edgedirectioncompactedgeendspaces, topologicalpathscycles}.

\begin{defn}
Let $G$ be a graph and $v\in V(G)$ a vertex. We will denote its edge class by $\mathcal{C}_v=[v]_E=\lbrace u\in V(G): u\sim_E v \rbrace$.

\end{defn}

We say that an edge class $\mathcal{C}_v$ of a vertex $v\in V(G)$ contains a ray $r\in \mathcal{R}(G)$ if $r$ passes through infinitely many elements of $\mathcal{C}_v$.

\begin{defn}
    Let $G$ be a graph and $v\in V(G)$. We say that $v$ edge-dominates a ray $r$ of $G$ if for every finite set of edges $X\subset E(G)$, $r$ has a tail in the same connected component of $v$ in $G\setminus X$.
\end{defn}

In \cite{edge-ends}, the following definition of when a graph is end-correlated is presented:

\begin{defn}
    A graph $G$ is said to be end-correlated if any two rays of $G$ that are edge-equivalent and non-dominated are vertex-equivalent.
\end{defn}

Using this definition, it is proved in Theorem 1 in \cite{edge-ends} that a graph $G$ being end-correlated is equivalent to two other combinatorial properties, namely: $G$ is not a caterpillar, and $G$ has the symmetric domination property. In our main theorem, in particular, we will prove two new equivalences with the property of being end-correlated.

Before we prove the main theorem, note that characterizing when $f_E=f\circ \phi$ is injective is different from characterizing when the end and edge-end spaces coincide. See the example below.

\begin{ex}
Note that in the graph $G$ from Figure \ref{fig3}, just like in the graph from Figure \ref{não injetivo}, we have $\vert \Omega(G)\vert =2$ and $\vert\Omega_E(G)\vert=1$, but in this example, $\Omega'(G)=\Omega_E(G)$, thus $\vert\Omega'(G)\vert =1$. Note also that $f_E:\Omega'(G)\longrightarrow \Omega_E(G)$ is a bijection while $f:\Omega(G)\longrightarrow\Omega_E(G)$ is not injective.

    \begin{figure}[ht]
        \centering

\tikzset{every picture/.style={line width=0.75pt}}   

\begin{tikzpicture}[x=0.6pt,y=0.6pt,yscale=-1,xscale=1]

%Straight Lines [id:da36261628921800115] 
\draw    (171.33,47.98) -- (193,47.83) ;
\draw [shift={(168.33,48)}, rotate = 359.61] [fill={rgb, 255:red, 0; green, 0; blue, 0 }  ][line width=0.08]  [draw opacity=0] (8.93,-4.29) -- (0,0) -- (8.93,4.29) -- cycle    ;
%Straight Lines [id:da14905344416748234] 
\draw    (193,47.83) -- (217.67,47.67) ;
\draw [shift={(217.67,47.67)}, rotate = 359.61] [color={rgb, 255:red, 0; green, 0; blue, 0 }  ][fill={rgb, 255:red, 0; green, 0; blue, 0 }  ][line width=0.75]      (0, 0) circle [x radius= 2.01, y radius= 2.01]   ;
\draw [shift={(193,47.83)}, rotate = 359.61] [color={rgb, 255:red, 0; green, 0; blue, 0 }  ][fill={rgb, 255:red, 0; green, 0; blue, 0 }  ][line width=0.75]      (0, 0) circle [x radius= 2.01, y radius= 2.01]   ;
%Straight Lines [id:da582589618790029] 
\draw    (217.67,47.67) -- (242.33,47.5) ;
%Straight Lines [id:da11876948802682064] 
\draw    (242.33,47.5) -- (267,47.33) ;
\draw [shift={(242.33,47.5)}, rotate = 359.61] [color={rgb, 255:red, 0; green, 0; blue, 0 }  ][fill={rgb, 255:red, 0; green, 0; blue, 0 }  ][line width=0.75]      (0, 0) circle [x radius= 2.01, y radius= 2.01]   ;
%Straight Lines [id:da5307292996997093] 
\draw    (267,47.33) -- (291.67,47.17) ;
\draw [shift={(291.67,47.17)}, rotate = 359.61] [color={rgb, 255:red, 0; green, 0; blue, 0 }  ][fill={rgb, 255:red, 0; green, 0; blue, 0 }  ][line width=0.75]      (0, 0) circle [x radius= 2.01, y radius= 2.01]   ;
\draw [shift={(267,47.33)}, rotate = 359.61] [color={rgb, 255:red, 0; green, 0; blue, 0 }  ][fill={rgb, 255:red, 0; green, 0; blue, 0 }  ][line width=0.75]      (0, 0) circle [x radius= 2.01, y radius= 2.01]   ;
%Straight Lines [id:da25076220629115686] 
\draw    (291.67,47.17) -- (316.33,47) ;
\draw [shift={(316.33,47)}, rotate = 359.61] [color={rgb, 255:red, 0; green, 0; blue, 0 }  ][fill={rgb, 255:red, 0; green, 0; blue, 0 }  ][line width=0.75]      (0, 0) circle [x radius= 2.01, y radius= 2.01]   ;
%Straight Lines [id:da39639238836943336] 
\draw    (171.33,92.98) -- (193,92.83) ;
\draw [shift={(168.33,93)}, rotate = 359.61] [fill={rgb, 255:red, 0; green, 0; blue, 0 }  ][line width=0.08]  [draw opacity=0] (8.93,-4.29) -- (0,0) -- (8.93,4.29) -- cycle    ;
%Straight Lines [id:da37517293203679813] 
\draw    (193,92.83) -- (217.67,92.67) ;
\draw [shift={(217.67,92.67)}, rotate = 359.61] [color={rgb, 255:red, 0; green, 0; blue, 0 }  ][fill={rgb, 255:red, 0; green, 0; blue, 0 }  ][line width=0.75]      (0, 0) circle [x radius= 2.01, y radius= 2.01]   ;
\draw [shift={(193,92.83)}, rotate = 359.61] [color={rgb, 255:red, 0; green, 0; blue, 0 }  ][fill={rgb, 255:red, 0; green, 0; blue, 0 }  ][line width=0.75]      (0, 0) circle [x radius= 2.01, y radius= 2.01]   ;
%Straight Lines [id:da5448989337543276] 
\draw    (217.67,92.67) -- (242.33,92.5) ;
%Straight Lines [id:da9201978449260212] 
\draw    (242.33,92.5) -- (267,92.33) ;
\draw [shift={(242.33,92.5)}, rotate = 359.61] [color={rgb, 255:red, 0; green, 0; blue, 0 }  ][fill={rgb, 255:red, 0; green, 0; blue, 0 }  ][line width=0.75]      (0, 0) circle [x radius= 2.01, y radius= 2.01]   ;
%Straight Lines [id:da2963575325890063] 
\draw    (267,92.33) -- (291.67,92.17) ;
\draw [shift={(291.67,92.17)}, rotate = 359.61] [color={rgb, 255:red, 0; green, 0; blue, 0 }  ][fill={rgb, 255:red, 0; green, 0; blue, 0 }  ][line width=0.75]      (0, 0) circle [x radius= 2.01, y radius= 2.01]   ;
\draw [shift={(267,92.33)}, rotate = 359.61] [color={rgb, 255:red, 0; green, 0; blue, 0 }  ][fill={rgb, 255:red, 0; green, 0; blue, 0 }  ][line width=0.75]      (0, 0) circle [x radius= 2.01, y radius= 2.01]   ;
%Straight Lines [id:da9528968724141348] 
\draw    (291.67,92.17) -- (316.33,92) ;
\draw [shift={(316.33,92)}, rotate = 359.61] [color={rgb, 255:red, 0; green, 0; blue, 0 }  ][fill={rgb, 255:red, 0; green, 0; blue, 0 }  ][line width=0.75]      (0, 0) circle [x radius= 2.01, y radius= 2.01]   ;
%Straight Lines [id:da7144013712172975] 
\draw    (316.33,47) -- (316.33,92) ;
%Straight Lines [id:da6390583396632779] 
\draw    (291.67,47.17) -- (291.67,92.17) ;
%Straight Lines [id:da6639420671895112] 
\draw    (267,47.33) -- (267,92.33) ;
%Straight Lines [id:da2667862540935686] 
\draw    (242.33,47.5) -- (242.33,92.5) ;
%Straight Lines [id:da605314787130963] 
\draw    (217.67,47.67) -- (217.67,92.67) ;
%Straight Lines [id:da7206370116950308] 
\draw    (193,47.83) -- (193,92.83) ;
%Curve Lines [id:da43603686112021345] 
\draw    (291.67,47.17) .. controls (292.33,46.5) and (303.33,37) .. (316.33,47) ;
%Curve Lines [id:da9998253771033485] 
\draw    (291.67,47.17) .. controls (290.83,43) and (306.33,27) .. (316.33,47) ;
%Curve Lines [id:da9692478666338079] 
\draw    (291.67,47.17) .. controls (289.33,40.5) and (307.83,18.5) .. (316.33,47) ;
%Curve Lines [id:da7926295759351148] 
\draw    (291.67,47.17) .. controls (289.33,40.5) and (305.33,7) .. (316.33,47) ;
%Straight Lines [id:da590252953500889] 
\draw  [dash pattern={on 0.84pt off 2.51pt}]  (303.33,14) -- (303.33,24) ;
%Straight Lines [id:da201605252192826] 
\draw    (266.67,47.17) -- (291.33,47) ;
%Curve Lines [id:da9825925077562615] 
\draw    (266.67,47.17) .. controls (267.33,46.5) and (278.33,37) .. (291.33,47) ;
%Curve Lines [id:da1474692736949773] 
\draw    (266.67,47.17) .. controls (265.83,43) and (281.33,27) .. (291.33,47) ;
%Curve Lines [id:da08664927277538825] 
\draw    (266.67,47.17) .. controls (264.33,40.5) and (282.83,18.5) .. (291.33,47) ;
%Curve Lines [id:da6342471324318424] 
\draw    (266.67,47.17) .. controls (264.33,40.5) and (280.33,7) .. (291.33,47) ;
%Straight Lines [id:da4934816109675154] 
\draw  [dash pattern={on 0.84pt off 2.51pt}]  (278.33,14) -- (278.33,24) ;
%Straight Lines [id:da09610089872277527] 
\draw    (243.17,48.17) -- (267.83,48) ;
%Curve Lines [id:da30670453994182745] 
\draw    (243.17,48.17) .. controls (243.83,47.5) and (254.83,38) .. (267.83,48) ;
%Curve Lines [id:da5926755376415206] 
\draw    (243.17,48.17) .. controls (242.33,44) and (257.83,28) .. (267.83,48) ;
%Curve Lines [id:da10183927565216011] 
\draw    (243.17,48.17) .. controls (240.83,41.5) and (259.33,19.5) .. (267.83,48) ;
%Curve Lines [id:da7776062836542526] 
\draw    (243.17,48.17) .. controls (240.83,41.5) and (256.83,8) .. (267.83,48) ;
%Straight Lines [id:da47434003857257] 
\draw  [dash pattern={on 0.84pt off 2.51pt}]  (254.83,15) -- (254.83,25) ;
%Straight Lines [id:da5514012814229587] 
\draw    (217.67,47.67) -- (242.33,47.5) ;
%Curve Lines [id:da696490201189019] 
\draw    (217.67,47.67) .. controls (218.33,47) and (229.33,37.5) .. (242.33,47.5) ;
%Curve Lines [id:da42897156751956667] 
\draw    (217.67,47.67) .. controls (216.83,43.5) and (232.33,27.5) .. (242.33,47.5) ;
%Curve Lines [id:da8367672011558364] 
\draw    (217.67,47.67) .. controls (215.33,41) and (233.83,19) .. (242.33,47.5) ;
%Curve Lines [id:da098499969958327] 
\draw    (217.67,47.67) .. controls (215.33,41) and (231.33,7.5) .. (242.33,47.5) ;
%Straight Lines [id:da9914287565316805] 
\draw  [dash pattern={on 0.84pt off 2.51pt}]  (229.33,14.5) -- (229.33,24.5) ;
%Straight Lines [id:da05085801053150396] 
\draw    (193.67,48.17) -- (218.33,48) ;
%Curve Lines [id:da43757814206428813] 
\draw    (193.67,48.17) .. controls (194.33,47.5) and (205.33,38) .. (218.33,48) ;
%Curve Lines [id:da3298268511128847] 
\draw    (193.67,48.17) .. controls (192.83,44) and (208.33,28) .. (218.33,48) ;
%Curve Lines [id:da11432459633014247] 
\draw    (193.67,48.17) .. controls (191.33,41.5) and (209.83,19.5) .. (218.33,48) ;
%Curve Lines [id:da872454117261296] 
\draw    (193.67,48.17) .. controls (191.33,41.5) and (207.33,8) .. (218.33,48) ;
%Straight Lines [id:da20036064895966998] 
\draw  [dash pattern={on 0.84pt off 2.51pt}]  (205.33,15) -- (205.33,25) ;
%Straight Lines [id:da6915505365958583] 
\draw    (497.33,91.52) -- (475.67,91.67) ;
\draw [shift={(500.33,91.5)}, rotate = 179.61] [fill={rgb, 255:red, 0; green, 0; blue, 0 }  ][line width=0.08]  [draw opacity=0] (8.93,-4.29) -- (0,0) -- (8.93,4.29) -- cycle    ;
%Straight Lines [id:da7359061530582391] 
\draw    (475.67,91.67) -- (451,91.83) ;
\draw [shift={(451,91.83)}, rotate = 179.61] [color={rgb, 255:red, 0; green, 0; blue, 0 }  ][fill={rgb, 255:red, 0; green, 0; blue, 0 }  ][line width=0.75]      (0, 0) circle [x radius= 2.01, y radius= 2.01]   ;
\draw [shift={(475.67,91.67)}, rotate = 179.61] [color={rgb, 255:red, 0; green, 0; blue, 0 }  ][fill={rgb, 255:red, 0; green, 0; blue, 0 }  ][line width=0.75]      (0, 0) circle [x radius= 2.01, y radius= 2.01]   ;
%Straight Lines [id:da4668540334943476] 
\draw    (451,91.83) -- (426.33,92) ;
%Straight Lines [id:da1275178097466969] 
\draw    (426.33,92) -- (401.67,92.17) ;
\draw [shift={(426.33,92)}, rotate = 179.61] [color={rgb, 255:red, 0; green, 0; blue, 0 }  ][fill={rgb, 255:red, 0; green, 0; blue, 0 }  ][line width=0.75]      (0, 0) circle [x radius= 2.01, y radius= 2.01]   ;
%Straight Lines [id:da11223677341530569] 
\draw    (401.67,92.17) -- (377,92.33) ;
\draw [shift={(377,92.33)}, rotate = 179.61] [color={rgb, 255:red, 0; green, 0; blue, 0 }  ][fill={rgb, 255:red, 0; green, 0; blue, 0 }  ][line width=0.75]      (0, 0) circle [x radius= 2.01, y radius= 2.01]   ;
\draw [shift={(401.67,92.17)}, rotate = 179.61] [color={rgb, 255:red, 0; green, 0; blue, 0 }  ][fill={rgb, 255:red, 0; green, 0; blue, 0 }  ][line width=0.75]      (0, 0) circle [x radius= 2.01, y radius= 2.01]   ;
%Straight Lines [id:da14680351016154325] 
\draw    (377,92.33) -- (352.33,92.5) ;
\draw [shift={(352.33,92.5)}, rotate = 179.61] [color={rgb, 255:red, 0; green, 0; blue, 0 }  ][fill={rgb, 255:red, 0; green, 0; blue, 0 }  ][line width=0.75]      (0, 0) circle [x radius= 2.01, y radius= 2.01]   ;
%Straight Lines [id:da45066884646157257] 
\draw    (497.33,46.52) -- (475.67,46.67) ;
\draw [shift={(500.33,46.5)}, rotate = 179.61] [fill={rgb, 255:red, 0; green, 0; blue, 0 }  ][line width=0.08]  [draw opacity=0] (8.93,-4.29) -- (0,0) -- (8.93,4.29) -- cycle    ;
%Straight Lines [id:da009465495429349269] 
\draw    (475.67,46.67) -- (451,46.83) ;
\draw [shift={(451,46.83)}, rotate = 179.61] [color={rgb, 255:red, 0; green, 0; blue, 0 }  ][fill={rgb, 255:red, 0; green, 0; blue, 0 }  ][line width=0.75]      (0, 0) circle [x radius= 2.01, y radius= 2.01]   ;
\draw [shift={(475.67,46.67)}, rotate = 179.61] [color={rgb, 255:red, 0; green, 0; blue, 0 }  ][fill={rgb, 255:red, 0; green, 0; blue, 0 }  ][line width=0.75]      (0, 0) circle [x radius= 2.01, y radius= 2.01]   ;
%Straight Lines [id:da08582930490146923] 
\draw    (451,46.83) -- (426.33,47) ;
%Straight Lines [id:da24228002576757024] 
\draw    (426.33,47) -- (401.67,47.17) ;
\draw [shift={(426.33,47)}, rotate = 179.61] [color={rgb, 255:red, 0; green, 0; blue, 0 }  ][fill={rgb, 255:red, 0; green, 0; blue, 0 }  ][line width=0.75]      (0, 0) circle [x radius= 2.01, y radius= 2.01]   ;
%Straight Lines [id:da10800733433406884] 
\draw    (401.67,47.17) -- (377,47.33) ;
\draw [shift={(377,47.33)}, rotate = 179.61] [color={rgb, 255:red, 0; green, 0; blue, 0 }  ][fill={rgb, 255:red, 0; green, 0; blue, 0 }  ][line width=0.75]      (0, 0) circle [x radius= 2.01, y radius= 2.01]   ;
\draw [shift={(401.67,47.17)}, rotate = 179.61] [color={rgb, 255:red, 0; green, 0; blue, 0 }  ][fill={rgb, 255:red, 0; green, 0; blue, 0 }  ][line width=0.75]      (0, 0) circle [x radius= 2.01, y radius= 2.01]   ;
%Straight Lines [id:da37667051227826975] 
\draw    (377,47.33) -- (352.33,47.5) ;
\draw [shift={(352.33,47.5)}, rotate = 179.61] [color={rgb, 255:red, 0; green, 0; blue, 0 }  ][fill={rgb, 255:red, 0; green, 0; blue, 0 }  ][line width=0.75]      (0, 0) circle [x radius= 2.01, y radius= 2.01]   ;
%Straight Lines [id:da24093476235975397] 
\draw    (352.33,92.5) -- (352.33,47.5) ;
%Straight Lines [id:da718096603678754] 
\draw    (377,92.33) -- (377,47.33) ;
%Straight Lines [id:da03946453350880064] 
\draw    (401.67,92.17) -- (401.67,47.17) ;
%Straight Lines [id:da18923789902839272] 
\draw    (426.33,92) -- (426.33,47) ;
%Straight Lines [id:da5232819018973228] 
\draw    (451,91.83) -- (451,46.83) ;
%Straight Lines [id:da47948513144361793] 
\draw    (475.67,91.67) -- (475.67,46.67) ;
%Straight Lines [id:da21436378809938206] 
\draw    (402,92.33) -- (377.33,92.5) ;
%Straight Lines [id:da9202454871448351] 
\draw    (425.5,91.33) -- (400.83,91.5) ;
%Straight Lines [id:da9178276320509453] 
\draw    (451,91.83) -- (426.33,92) ;
%Straight Lines [id:da005836844497847382] 
\draw    (475,91.33) -- (450.33,91.5) ;
%Straight Lines [id:da9900270807454029] 
\draw    (352.27,47.05) -- (376.93,46.99) ;
%Straight Lines [id:da3158252753142138] 
\draw    (376.93,46.99) -- (401.6,46.93) ;
%Straight Lines [id:da5097440674760058] 
\draw    (376.93,46.99) -- (401.6,46.93) ;
%Straight Lines [id:da5913313110971785] 
\draw    (352.93,47.39) -- (377.6,47.33) ;
%Curve Lines [id:da4310904623630387] 
\draw    (352.93,47.39) .. controls (353.6,46.72) and (387.67,30.33) .. (402.27,46.55) ;
%Curve Lines [id:da48418206106301953] 
\draw    (352.93,47.39) .. controls (356.33,37) and (413,24.33) .. (426.93,46.49) ;
%Curve Lines [id:da7903728336050316] 
\draw    (352.93,47.39) .. controls (359,33.67) and (423.67,13) .. (451.6,46.43) ;
%Curve Lines [id:da4890693186707278] 
\draw    (352.93,47.39) .. controls (357.67,25) and (439.67,11.67) .. (475.67,46.67) ;
%Straight Lines [id:da2722731767381704] 
\draw  [dash pattern={on 0.84pt off 2.51pt}]  (481,23.67) -- (458.03,27.61) ;
%Straight Lines [id:da37243977244194004] 
\draw    (402.27,46.55) -- (426.93,46.49) ;
%Straight Lines [id:da9806772468455298] 
\draw    (426.93,46.49) -- (451.6,46.43) ;
%Straight Lines [id:da07422205402322402] 
\draw    (426.93,46.49) -- (451.6,46.43) ;
%Straight Lines [id:da8348888718309961] 
\draw    (402.93,46.89) -- (427.6,46.83) ;
%Straight Lines [id:da4414592985523118] 
\draw    (450.27,46.55) -- (474.93,46.49) ;
%Straight Lines [id:da8699911773749033] 
\draw    (450.93,46.89) -- (475.6,46.83) ;
%Straight Lines [id:da24181717855939333] 
\draw    (316.77,47.55) -- (350.91,47.49) ;
%Straight Lines [id:da7186222002627661] 
\draw    (317.69,47.89) -- (351.83,47.83) ;
%Curve Lines [id:da9985287736925922] 
\draw    (317.69,47.89) .. controls (318.61,47.22) and (333.9,37.77) .. (351.83,47.83) ;
%Curve Lines [id:da6199623154788488] 
\draw    (317.69,47.89) .. controls (316.56,43.72) and (338.11,27.79) .. (351.83,47.83) ;
%Curve Lines [id:da2582278415525656] 
\draw    (317.69,47.89) .. controls (314.5,41.21) and (340.24,19.29) .. (351.83,47.83) ;
%Curve Lines [id:da5094709837617235] 
\draw    (317.69,47.89) .. controls (314.5,41.21) and (336.85,7.78) .. (351.83,47.83) ;
%Straight Lines [id:da9813688077503328] 
\draw  [dash pattern={on 0.84pt off 2.51pt}]  (334.04,14.77) -- (333.98,24.77) ;
%Straight Lines [id:da5608608679665656] 
\draw    (316.33,92) -- (352.33,92.5) ;

\end{tikzpicture}

        \caption{Graph where $f_E$ is injective, but $f$ is not.}

        \label{fig3}
    \end{figure}
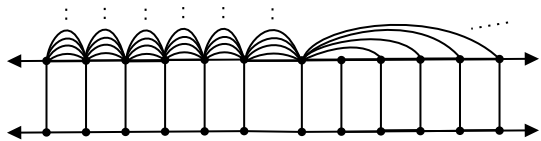
\end{ex}

Before presenting the main theorem, we state the well known Star-Comb Lemma, which will be used throughout its proof.

\begin{lemma}[Star-Comb lemma]
    Let $U$ be an infinite set of vertices in a connected graph $G$. Then $G$ contains either a comb attached to $U$ or a star attached to $U$.
\end{lemma}

\begin{thm}\label{theo1}
    Let $G$ be an infinite graph. The following statements are equivalent:
    \begin{enumerate}
        \item There exists a well-defined and injective function $f_E:\Omega'(G)\longrightarrow \Omega_E(G)$ such that the diagram below commutes.

        \begin{figure}[ht]
            \centering

\tikzset{every picture/.style={line width=0.75pt}}   

\begin{tikzpicture}[x=0.6pt,y=0.6pt,yscale=-1,xscale=1]

%Straight Lines [id:da1871348557637702] 
\draw    (260.5,54.6) -- (306.9,89.4) ;
\draw [shift={(308.5,90.6)}, rotate = 216.87] [color={rgb, 255:red, 0; green, 0; blue, 0 }  ][line width=0.75]    (10.93,-3.29) .. controls (6.95,-1.4) and (3.31,-0.3) .. (0,0) .. controls (3.31,0.3) and (6.95,1.4) .. (10.93,3.29)   ;
%Straight Lines [id:da8279285634918194] 
\draw    (280.5,36.6) -- (382.5,35.62) ;
\draw [shift={(384.5,35.6)}, rotate = 179.45] [color={rgb, 255:red, 0; green, 0; blue, 0 }  ][line width=0.75]    (10.93,-3.29) .. controls (6.95,-1.4) and (3.31,-0.3) .. (0,0) .. controls (3.31,0.3) and (6.95,1.4) .. (10.93,3.29)   ;
%Straight Lines [id:da9966586488343272] 
\draw    (362.5,92.6) -- (409.95,53.86) ;
\draw [shift={(411.5,52.6)}, rotate = 140.77] [color={rgb, 255:red, 0; green, 0; blue, 0 }  ][line width=0.75]    (10.93,-3.29) .. controls (6.95,-1.4) and (3.31,-0.3) .. (0,0) .. controls (3.31,0.3) and (6.95,1.4) .. (10.93,3.29)   ;
%Curve Lines [id:da47978489430037397] 
\draw    (322.5,47.6) .. controls (290.82,90.17) and (364.01,87.66) .. (341.22,48.79) ;
\draw [shift={(340.5,47.6)}, rotate = 57.99] [color={rgb, 255:red, 0; green, 0; blue, 0 }  ][line width=0.75]    (10.93,-3.29) .. controls (6.95,-1.4) and (3.31,-0.3) .. (0,0) .. controls (3.31,0.3) and (6.95,1.4) .. (10.93,3.29)   ;

% Text Node
\draw (317,84.4) node [anchor=north west][inner sep=0.75pt]    {$\Omega ( G)$};
% Text Node
\draw (231,28.4) node [anchor=north west][inner sep=0.75pt]    {$\Omega '( G)$};
% Text Node
\draw (394,29.4) node [anchor=north west][inner sep=0.75pt]    {$\Omega _{E}( G)$};
% Text Node
\draw (317,11.4) node [anchor=north west][inner sep=0.75pt]    {$f_{E}$};
% Text Node
\draw (258,69.4) node [anchor=north west][inner sep=0.75pt]    {$\phi$};
% Text Node
\draw (393,68.4) node [anchor=north west][inner sep=0.75pt]    {$f $};

\end{tikzpicture}
        \end{figure}

\item $G$ is end-correlated

\item For every finite set $F\subset V(G)$ and every vertex $v\in V(G)$, there is at most one connected component $C$ of $G\setminus F$ such that the edge class $\mathcal{C}_v\cap C$ contains a non-dominated ray.

    \end{enumerate}
\end{thm}

\begin{proof}

$3\Rightarrow 2:$ Suppose, towards a contradiction, that there are two rays $r, s\in \mathcal{R}(G)$ that are non-dominated and not vertex-equivalent but are edge-equivalent.

Since $r$ and $s$ are not vertex-equivalent, there exists a finite set $V_1\subset V(G)$ which separates $r$ from $s$. As $r$ and $s$ are edge-equivalent, let $S_1\subset V_1$ be the set of all vertices that edge-dominates $r$ and let $L_1\subset V_1$ be the set of all vertices that edge-dominates $s$.

\begin{claim}
    $L_1\cap S_1\neq \emptyset$.
\end{claim}

\begin{proof}
    Suppose, towards a contradiction, that $L_1\cap S_1=\emptyset$. Given $u\in V_1$, then $u$ edge-dominates at most one of the rays $r$ and $s$.
  Let $A_r=V_1\setminus S_1$ and $A_s~=~V_1\setminus L_1$. For each $u\in A_r$, consider $E_u$ a finite set of edges that separates $u$ from $r$. For each $u\in A_s$, consider $F_u$ the finite set of edges that separates $u$ from $s$. Then, let $F=(\bigcup_{u\in A_r} E_u)\cup(\bigcup_{u\in A_s} F_u)$.

    Since $F$ is a finite set of edges and $r\sim_E s$, there exists a path $P$ in $G\setminus F$ connecting $r$ to $s$. As $V_1$ separates $r$ and $s$, we have $P$ passes through $V_1$, but then $P$ must pass through $L_1$, because $F$ separates $s$ from $A_s=V_1\setminus L_1$. However, there is no path in $G\setminus F$ from $L_1$ to $r$, since $L_1\subset A_r=V_1\setminus S_1$. Therefore, $P$ cannot exist, which leads to a contradiction because $r\sim_E s$.
    
\end{proof}

Consider, then, $A=S_1\cap L_1\neq\emptyset$. Thus, if $v\in A$, $v$ edge-dominates both $r$ and $s$. Since $r$ and $s$ are non-dominated, there exists a finite set of vertices $V_2\subset V(G)$ that separates $V_1$ from $r$. As $V_2$ separates $V_1$ from $r$, it also separates $A$ from $r$. Furthermore, $V_2$ separates $r$ from $s$. Thus, since $r$ and $s$ are edge-equivalent, the set $A_2\subset V_2$ of all vertices that edge-dominates $r$ is non-empty.

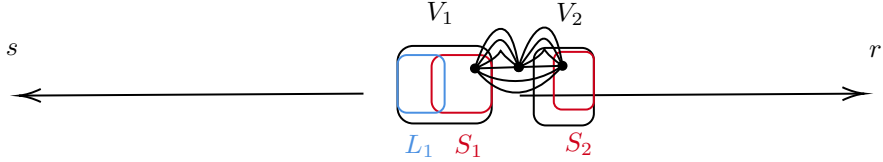
\begin{figure}[ht]
    \centering

\tikzset{every picture/.style={line width=0.75pt}}   

\begin{tikzpicture}[x=0.75pt,y=0.75pt,yscale=-1,xscale=1]

%Rounded Rect [id:dp6909567257592978] 
\draw   (306.67,39.47) .. controls (306.67,35.16) and (310.16,31.67) .. (314.47,31.67) -- (345.87,31.67) .. controls (350.17,31.67) and (353.67,35.16) .. (353.67,39.47) -- (353.67,62.87) .. controls (353.67,67.17) and (350.17,70.67) .. (345.87,70.67) -- (314.47,70.67) .. controls (310.16,70.67) and (306.67,67.17) .. (306.67,62.87) -- cycle ;
%Straight Lines [id:da34952121701253014] 
\draw    (367.67,56.67) -- (537.67,55.68) ;
\draw [shift={(539.67,55.67)}, rotate = 179.67] [color={rgb, 255:red, 0; green, 0; blue, 0 }  ][line width=0.75]    (10.93,-3.29) .. controls (6.95,-1.4) and (3.31,-0.3) .. (0,0) .. controls (3.31,0.3) and (6.95,1.4) .. (10.93,3.29)   ;
%Straight Lines [id:da7267304633926029] 
\draw    (289.67,55.67) -- (119.67,56.66) ;
\draw [shift={(117.67,56.67)}, rotate = 359.67] [color={rgb, 255:red, 0; green, 0; blue, 0 }  ][line width=0.75]    (10.93,-3.29) .. controls (6.95,-1.4) and (3.31,-0.3) .. (0,0) .. controls (3.31,0.3) and (6.95,1.4) .. (10.93,3.29)   ;
%Rounded Rect [id:dp056155344284040454] 
\draw  [color={rgb, 255:red, 208; green, 2; blue, 27 }  ,draw opacity=1 ] (323.67,42.07) .. controls (323.67,38.86) and (326.26,36.27) .. (329.47,36.27) -- (347.87,36.27) .. controls (351.07,36.27) and (353.67,38.86) .. (353.67,42.07) -- (353.67,59.47) .. controls (353.67,62.67) and (351.07,65.27) .. (347.87,65.27) -- (329.47,65.27) .. controls (326.26,65.27) and (323.67,62.67) .. (323.67,59.47) -- cycle ;
%Rounded Rect [id:dp3975096783349634] 
\draw  [color={rgb, 255:red, 74; green, 144; blue, 226 }  ,draw opacity=1 ] (306.67,40.97) .. controls (306.67,38.37) and (308.77,36.27) .. (311.37,36.27) -- (325.47,36.27) .. controls (328.06,36.27) and (330.17,38.37) .. (330.17,40.97) -- (330.17,60.57) .. controls (330.17,63.16) and (328.06,65.27) .. (325.47,65.27) -- (311.37,65.27) .. controls (308.77,65.27) and (306.67,63.16) .. (306.67,60.57) -- cycle ;
%Rounded Rect [id:dp17469129538708283] 
\draw   (374.67,38.67) .. controls (374.67,35.35) and (377.35,32.67) .. (380.67,32.67) -- (398.67,32.67) .. controls (401.98,32.67) and (404.67,35.35) .. (404.67,38.67) -- (404.67,65.67) .. controls (404.67,68.98) and (401.98,71.67) .. (398.67,71.67) -- (380.67,71.67) .. controls (377.35,71.67) and (374.67,68.98) .. (374.67,65.67) -- cycle ;
%Rounded Rect [id:dp00819889784935568] 
\draw  [color={rgb, 255:red, 208; green, 2; blue, 27 }  ,draw opacity=1 ] (384.67,38.67) .. controls (384.67,36.46) and (386.46,34.67) .. (388.67,34.67) -- (400.67,34.67) .. controls (402.88,34.67) and (404.67,36.46) .. (404.67,38.67) -- (404.67,59.67) .. controls (404.67,61.88) and (402.88,63.67) .. (400.67,63.67) -- (388.67,63.67) .. controls (386.46,63.67) and (384.67,61.88) .. (384.67,59.67) -- cycle ;
%Straight Lines [id:da8283995617074176] 
\draw    (345.33,43) -- (367.17,42.33) ;
\draw [shift={(367.17,42.33)}, rotate = 358.25] [color={rgb, 255:red, 0; green, 0; blue, 0 }  ][fill={rgb, 255:red, 0; green, 0; blue, 0 }  ][line width=0.75]      (0, 0) circle [x radius= 2.01, y radius= 2.01]   ;
\draw [shift={(345.33,43)}, rotate = 358.25] [color={rgb, 255:red, 0; green, 0; blue, 0 }  ][fill={rgb, 255:red, 0; green, 0; blue, 0 }  ][line width=0.75]      (0, 0) circle [x radius= 2.01, y radius= 2.01]   ;
%Straight Lines [id:da5894028319319387] 
\draw    (367.17,42.33) -- (389,41.67) ;
\draw [shift={(389,41.67)}, rotate = 358.25] [color={rgb, 255:red, 0; green, 0; blue, 0 }  ][fill={rgb, 255:red, 0; green, 0; blue, 0 }  ][line width=0.75]      (0, 0) circle [x radius= 2.01, y radius= 2.01]   ;
\draw [shift={(367.17,42.33)}, rotate = 358.25] [color={rgb, 255:red, 0; green, 0; blue, 0 }  ][fill={rgb, 255:red, 0; green, 0; blue, 0 }  ][line width=0.75]      (0, 0) circle [x radius= 2.01, y radius= 2.01]   ;
%Curve Lines [id:da32497519814375797] 
\draw    (345.33,43) .. controls (368.33,33) and (349,29.67) .. (367.17,42.33) ;
%Curve Lines [id:da39795645682649816] 
\draw    (345.33,43) .. controls (348.33,47) and (369,55.67) .. (389,41.67) ;
%Curve Lines [id:da775191273984541] 
\draw    (345.33,43) .. controls (348.33,47) and (365.67,69) .. (389,41.67) ;
%Curve Lines [id:da3687046106948798] 
\draw    (345.33,43) .. controls (357,29) and (357,19.67) .. (367.17,42.33) ;
%Curve Lines [id:da39741952515258694] 
\draw    (345.33,43) .. controls (351,23.67) and (362.33,10.33) .. (367.17,42.33) ;
%Curve Lines [id:da8104238606682849] 
\draw    (367.17,42.33) .. controls (390.17,32.33) and (370.83,29) .. (389,41.67) ;
%Curve Lines [id:da39774317181443275] 
\draw    (367.17,42.33) .. controls (378.83,28.33) and (378.83,19) .. (389,41.67) ;
%Curve Lines [id:da613894344999479] 
\draw    (367.17,42.33) .. controls (372.83,23) and (384.17,9.67) .. (389,41.67) ;

% Text Node
\draw (319.67,8.07) node [anchor=north west][inner sep=0.75pt]    {$V_{1}$};
% Text Node
\draw (540.67,30.07) node [anchor=north west][inner sep=0.75pt]    {$r$};
% Text Node
\draw (109.67,29.07) node [anchor=north west][inner sep=0.75pt]    {$s$};
% Text Node
\draw (333.67,75.07) node [anchor=north west][inner sep=0.75pt]  [color={rgb, 255:red, 208; green, 2; blue, 27 }  ,opacity=1 ]  {$S_{1}$};
% Text Node
\draw (308.67,75.07) node [anchor=north west][inner sep=0.75pt]  [color={rgb, 255:red, 74; green, 144; blue, 226 }  ,opacity=1 ]  {$L_{1}$};
% Text Node
\draw (384.33,8.4) node [anchor=north west][inner sep=0.75pt]    {$V_{2}$};
% Text Node
\draw (388.67,74.07) node [anchor=north west][inner sep=0.75pt]  [color={rgb, 255:red, 208; green, 2; blue, 27 }  ,opacity=1 ]  {$S_{2}$};

\end{tikzpicture}
    \caption{Representation}
    \label{fig:enter-label}
\end{figure}
\begin{claim}\label{claim1}
 If $w\in A$ and $u\in A_2$, then $w\sim_E u$.

\end{claim}

\begin{proof}
Indeed, since $w\in A$ and $u\in A_2$, then $w$  and $u$ edge-dominates $r$. Let $F\subset E(G)$ be a finite set and let $\overline{r}$ be a tail of $r$ in $G\setminus F$. Then $w$ and $u$ are in the same connected component of $\overline{r}$ of the graph $G\setminus F$.
  
\end{proof}

Inductively, for every $n\in \mathbb{N}$, $V_n\subset V(G)$ does non-dominate $r$, since $V_n$ is finite and $r$ is non-dominated. Therefore, there exists a finite set $V_{n+1}\subset V(G)$ that separates $V_n$ from $r$. Let $A_{n+1}$ be the subset of $V_{n+1}$ that edge-dominates $r$. Using the same argument from Claim \ref{claim1}, we can obtain that each vertex of $A_{n+1}$ is edge-equivalent with every vertex of $A_n\subset V_n$.

Thus, we construct an infinite family of finite sets of vertices $\lbrace A_1, A_2, A_3, \dots , A_n, \dots \rbrace$, with $A=A_1$, such that $A_i\cap A_j=\emptyset$ and each vertex $u\in \bigcup_{i\in\mathbb{N}} A_i$ edge-dominates $r$. Therefore, given distinct $i,j\in \mathbb{N}$, $u_i\in A_i$ and $u_j\in A_j$, then $u_i\sim_E u_j$.

\begin{claim}
    There exists a ray $t$ in $G$ such that $t\sim r$ and $t\cap A_n\neq \emptyset$ for infinitely many $n\in\mathbb{N}$.
 
\end{claim}
\begin{proof}
For each $n\in\mathbb{N}$, consider $C_n$ to be the connected component of $G\setminus V_n$ that contains $r$. Note that, if $V_{n+1}$ separates $V_n$ from $r$, then $V_{n+1}\subset V(C_n)$. Since $r$ is a non-dominated end, $r$ determines a topological end, and therefore $\bigcap_{n\in\mathbb{N}} C_n=\emptyset$.

Let $v_1\in A_1$, $v_3\in A_3$, and let $P_1$ be a path connecting $v_1$ to $v_3$ without passing through $(V_1\setminus A_1)\cup (V_2\setminus A_2)$. It is possible to find such a path because $(V_1\setminus A_1)\cup (V_2\setminus A_2)$ is not edge-equivalent to $v_1$ and $v_3$. Note that $P_1$ must pass through some vertex $v_2\in A_2$, since $V_2$ separates $V_1$ from $V_3$. Consider $n_1=1$, $n_2=3$, and let $n_3$ be the successor of the smallest natural number $n$ such that $P_1\cap C_{n}=\emptyset$. Let $v_{n_3}\in A_{n_3}$. Consider $P_2$ to be a path connecting $v_3$ to $v_{n_3}$ without passing through $\bigcup_{j=1}^{n_3}(V_j\setminus A_j)$. Note that $P_2$ must pass through some vertex $v_{n_3-1}\in A_{n_3-1}$, since $V_{n_3-1}$ separates $V_3$ from $V_{n_3}$. Furthermore, it is possible to obtain $P_2\subset C_3$.

Suppose inductively that we have constructed paths $P_1, P_2, \dots, P_k$, where $P_i$ connects the vertex $v_{n_i}\in A_{n_i}$ to $v_{n_{i+1}}\in A_{n_{i+1}}$, passes through some vertex $v_{n_{i+1}-1}\in A_{n_{i+1}-1}$, and is contained in $C_{n_i}$. Then, consider $n_{k+2}$ to be the successor of the smallest natural number $n$ such that $P_k\cap C_{n}=\emptyset$. Let $v_{n_{k+2}}\in A_{n_{k+2}}$. Consider $P_{k+1}$ a path connecting $v_{n_{k+1}}$ to $v_{n_{k+2}}$ without passing through $\bigcup_{j=1}^{n_{k+2}}(V_j\setminus A_j)$. Note that $P_{k+1}$ must pass through some vertex $v_{n_{k+2}-1}\in A_{n_{k+2}-1}$, since $V_{n_{k+2}-1}$ separates $V_{n_{k+1}}$ from $V_{n_{k+2}}$. Furthermore, it is possible to obtain $P_{k+1}\subset C_{n_{k+1}}$.

    Since $Z=\bigcup_{n\in\mathbb{N}} P_n$ is a connected infinite graph and $U=\bigcup_{n\in\mathbb{N}} Z\cap A_n$ is infinite, then, by the Star-comb Lemma, there exists a star in $Z$ with leaves in $U$ or a comb in $Z$ with teeth in $U$. Suppose it is a star with center $u$ and leaves in $U$. Then, since $\bigcap_{n\in\mathbb{N}} C_n=\emptyset$ and $C_{n+1}\subset C_n$, there exists some $k\in\mathbb{N}$ such that $u\notin C_n$ for all $n\geq k$. Since $V_{k+1}$ separates $u$ from $C_{n}$ for all $n\geq k+1$, all paths between $u$ and $U$ must pass through $V_{k+1}$. However, $V_{k+1}$ is finite, and therefore, $u$ cannot be the center of a star with leaves in $U$. Thus, there must be a comb $P$ in $Z$ with spine $t$ and teeth in $U$. Since the paths $P_n$ are such that $P_n \cap V_{n_{k+1}-1}\subset A_{n_{k+1}-1}$, we have $t\cap A_n\neq \emptyset$ for infinitely many  $n\in\mathbb{N}$.

Finally, note that $t\sim r$. Indeed, consider $Q_1$ a path that connects $v_1$ to $r$. Let $n_1\in\mathbb{N}$ be the largest natural number such that $Q_1\cap V_n\neq \emptyset$. Then $V_{n_1+1}$ separates $r$ from $Q_1$. Let $u_2\in t\cap A_{n_2}$. Consider a path $Q_2$ from $u_2$ to $r$ in $G\setminus (\bigcup_{i=1}^{n_1+1} V_{i})$. Then, $Q_1\cap Q_2=\emptyset$. Let $n_2\in\mathbb{N}$ be the largest natural number such that $Q_2\cap V_n\neq \emptyset$. Suppose we have constructed $Q_1, Q_2, \dots Q_k$ such that $Q_i\cap Q_j=\emptyset$ for all $i,j\in \lbrace 1, 2, \dots , k\rbrace$ with $i\neq j$ and that each $Q_i$ connects $t$ to $r$. We will construct a path $Q_{k+1}$ which does not intersect the previous $Q_i$'s and connects $t$ to $r$. Let $n_k\in \mathbb{N}$ be the largest natural number such that $Q_i\cap V_n\neq \emptyset$ for some $i\in \lbrace 1, 2, \dots , k\rbrace$. Then $V_{n_k+1}$ separates $r$ from $Q_k, Q_{k-1}, \dots , Q_1$. Let $u_{k+1}\in t\cap A_{n_{k+1}}$. Consider a path $Q_{k+1}$ from $u_{k+1}$ to $r$ in $G\setminus (\bigcup_{i=1}^{n_k+1} V_{i})$. Then, $Q_{k+1}\cap Q_i=\emptyset$ for all $i\in \lbrace 1, 2 \dots , k\rbrace$. Thus, $t\sim r$.
     
\end{proof}

Note that, using an analogous construction, we can obtain a ray $t'$ starting at $A$, which passes through infinitely many vertices of $\mathcal{C}_v$ and is equivalent to $s$. Since $s$ and $r$ are non-dominated, $t$ and $t'$ are also non-dominated. Thus, let $C$ and $\overline{C}$ be connected components of $G\setminus V_1$, where $r$ lives in $C$ and $s$ lives in $\overline{C}$. Then $\mathcal{C}_v\cap C$ contains the non-dominated ray $t$ and $\mathcal{C}_v\cap \overline{C}$ contains the non-dominated ray $t'$. This leads to a contradiction.

$2\Rightarrow 3:$ Suppose, towards a contradiction, that statement 2 holds, but there exists a finite set $F\subset V(G)$ and a vertex $v\in V(G)$ such that there are two connected components $C$ and $C'$ of $G\setminus F$ with $\mathcal{C}_v\cap C$ and $\mathcal{C}_v\cap C'$ containing a non-dominated ray. Let $t$ be the non-dominated ray in $\mathcal{C}_v\cap C$ and $t'$ be the non-dominated ray in $\mathcal{C}_v\cap C'$. Then $t$ and $t'$ are not vertex-equivalent. Since $t\cap (\mathcal{C}_v\cap C)\neq\emptyset$ and $t'\cap (\mathcal{C}_v\cap C')\neq\emptyset$, there exist $u\in t\cap \mathcal{C}_v$ and $w\in t'\cap \mathcal{C}_v$. Since $u,w\in \mathcal{C}_v$, then for every finite set $L\subset E(G)$, there exists a path $P$ in $G\setminus L$ connecting $u$ to $w$. Therefore, $P$ connects $t$ to $t'$ in $G\setminus L$. Thus, $t$ and $t'$ are edge-equivalent, which leads to a contradiction.

$1\Rightarrow 2:$ Let $r,s\in \mathcal{R}(G)$ be two non-dominated and not vertex-equivalent rays. Then there exist two ends $\epsilon, \sigma\in \Omega'(G)$ such that $\phi(\epsilon)=[r]\in \Omega(G)$ and $\phi(\sigma)=[s]\in\Omega(G)$. Since $r\nsim s$, then $[r]\neq [s]$. Therefore, $\epsilon\neq \sigma$. As $f_E$ is injective, we have $f_E(\epsilon)\neq f_E(\sigma)$. So, $f([r])\neq f([s])$. Since $f([r])=[r]_E$ and $f([s])=[s]_E$, we have $[r]_E\neq [s]_E$. Thus, $r\nsim_E s$.

$2\Rightarrow 1:$ Consider the function $f_E:\Omega'(G)\longrightarrow \Omega_E(G)$ given by $f_E(\epsilon)=f(\phi(\epsilon))$. Given $\epsilon, \sigma \in \Omega'(G)$ with $\epsilon\neq \sigma$, then $\phi(\epsilon)\neq \phi(\sigma)$. Hence, $\phi(\epsilon)$ and $\phi(\sigma)$ are non-dominated and distinct ends. Let $r\in \phi(\epsilon)$ and $s\in \phi(\sigma)$, then $r\nsim s$. Since $r$ and $s$ are non-dominated, then by 2 we have $r\nsim_E s$. Since $f (\phi(\epsilon))=f([r])=[r]_E$ and $f (\phi(\sigma))=f([s])=[s]_E$, we have $f_E(\epsilon)\neq f_E(\sigma)$.
 
\end{proof}

\begin{defn}
    Let $G$ be an infinite graph. We say that an edge-end $\varepsilon\in \Omega_E(G)$ is almost non-dominated if there exist at least one ray $r\in \varepsilon$ such that $r$ is non-dominated. 
\end{defn}

\begin{prop}\label{prop1}
    Let $G$ be an infinite graph. The topological ends of $G$ are exactly the almost non-dominated edge-ends of $G$, that is, the image set of $f_E=f\circ \phi: \Omega'(G)\longrightarrow \Omega_E(G)$ is exactly the subset of edge-ends that are almost non-dominated.
   
\end{prop}
\begin{proof}
Indeed, given a topological end $\epsilon\in \Omega'(G)$, we have from Theorem 4.11 of the paper \cite{topologicalendsversuscombinatorialends} that $\phi(\epsilon)=\varepsilon\in \Omega(G)$ is a non-dominated end of $G$. Let $r\in \mathcal{R}(G)$ be such that $[r]=\varepsilon=\phi(\epsilon)$. Since $\varepsilon$ is non-dominated, $r$ is non-dominated. Therefore, $f([r])=[r]_E$ is an almost non-dominated edge-end.

Conversely, let $\varepsilon'\in \Omega_E(G)$ be an almost non-dominated edge-end. Then, there exists $r\in \mathcal{R}(G)$ such that $[r]_E=\varepsilon'$ and $r$ is non-dominated. Therefore, $[r]\in \Omega(G)$ is a non dominated end, where $f([r])=[r]_E=\varepsilon'$. Thus, there exists $\epsilon\in \Omega'(G)$ such that $\phi(\epsilon)=[r]$ and, therefore, $f(\phi(\epsilon))=\varepsilon'$.

\end{proof}

\section*{Acknowledgments}
The first named author thanks the support of Fundação de Amparo à Pesquisa do Estado de São Paulo (FAPESP), being sponsored through grant number 2025/12199-3. The second named author acknowledges the support of Conselho Nacional de Desenvolvimento Científico e Tecnológico (CNPq) through grant number 165761/2021-0. The third named author acknowledges the support of Conselho Nacional de Desenvolvimento Científico e Tecnológico (CNPq) through grant number 141373/2025-3. This study was financed in part
by the Coordenação de Aperfeiçoamento de Pessoal de Nível Superior – Brasil (CAPES) – Finance Code 001.

\bibliography{referencias}
\bibliographystyle{plain}

\Addresses

\end{document}